\documentclass[11pt,twocolumn]{article}
\usepackage[numbers,sort&compress]{natbib}
\usepackage{jabbrv}
\usepackage{multibib}
\newcites{S}{Supporting References} % Creates the \citeS{} command

\usepackage[
  letterpaper,
  top=1in,
  bottom=1in,
  left=0.65in,
  right=0.65in,
  columnsep=0.25in
]{geometry}
\usepackage[T1]{fontenc}
\usepackage[utf8]{inputenc}
\usepackage{lmodern}
\usepackage{microtype}

\usepackage{bm}
\usepackage{amsmath,amssymb,amsthm}
\usepackage{graphicx}
\usepackage{booktabs}
\usepackage{float}
\usepackage{placeins}
\usepackage{authblk}
\usepackage[colorlinks=true,allcolors=blue]{hyperref}

\usepackage{fancyhdr}

\fancypagestyle{firstpage}{
  \fancyhf{}
  \fancyfoot[C]{\footnotesize E-mail: \href{mailto:m.brolly@ed.ac.uk}{m.brolly@ed.ac.uk}}

}

\theoremstyle{plain}
\newtheorem{proposition}{Proposition}

\usepackage{libertine}
\usepackage[libertine]{newtxmath}

\usepackage{bm}
\usepackage{amsthm}
\usepackage{placeins} % Required for \FloatBarrier

\newcommand{\indep}{\mathrel{\perp\!\!\!\perp}}

\theoremstyle{plain}
\newtheorem{suppproposition}{Supplementary Proposition} 
\newtheorem{definition}{Definition}
\title{Stochasticity and probabilistic trajectory scoring are essential for data-driven closures of chaotic systems}

\author{Martin T. Brolly}
\affil{School of Mathematics and Maxwell Institute for Mathematical Sciences, University of Edinburgh, King's Buildings, Edinburgh EH9 3FD, UK}

\date{\today}

\begin{document}

\twocolumn[
\begin{@twocolumnfalse}

\maketitle

\begin{center}
\begin{minipage}{0.9\textwidth}
\begin{center}
\textbf{Abstract}
\end{center}
\small
Coarse-grained models of chaotic systems neglect unresolved degrees of freedom, inducing structured model error that limits predictability and distorts long-term statistics.
Typical data-driven closures are trained to minimize error over a single time step, implicitly assuming Markovian dynamics and often failing to capture long-term behavior. Recent approaches instead optimize losses over finite trajectories.
However, when such trajectory-based training is carried out with deterministic pointwise losses, it introduces a fundamental mathematical degeneracy.
We prove that optimizing pointwise deterministic losses such as mean squared error over chaotic trajectories suppresses predictive variance, with corresponding loss of physical variability in long integrations.
In contrast, strictly proper scoring rules avoid this degeneracy.
By targeting forecast distributions rather than realized trajectories, they remove the penalty against predictive spread and align the long-lead optimum with the invariant measure.
Using quasi-geostrophic turbulence as a canonical chaotic system, we validate this theory: one-step-trained closures fail to capture stable coarse-grained dynamics, while deterministic closures optimized over trajectories exhibit the variance-loss tendency predicted by our analysis.
Stochastic closures calibrated over trajectories using the energy score, however, overcome both structural limitations, yielding skillful ensemble forecasts and realistic long-term statistics. Our results establish that both stochastic modeling and trajectory-based calibration are essential for faithfully representing the dynamics of coarse-grained systems.
\end{minipage}
\end{center}

\vspace{1em}

\end{@twocolumnfalse}
]

% \noindent\textbf{Keywords:} forecasting; coarse-graining; reduced-order modeling; turbulence; probabilistic

% \section*{Significance statement}
% Simulating complex systems—whether in climate science, biology, or engineering—often requires ignoring certain components to make computations feasible. However, this simplification introduces errors that accumulate over time, causing models to drift away from reality. While conventional approaches focus on applying deterministic corrections to minimize immediate errors, we show that it is necessary to model these errors probabilistically and optimize predictions over longer trajectories. We demonstrate that randomness is not just useful for quantifying uncertainty, but is structurally required for calibrated forecasts and realistic long-term behavior. This work establishes a general framework for learning stable, statistically accurate reduced-order models in any setting where the full state of a chaotic system cannot be resolved.

% \section*{Competing interests}
% The author declares no competing interests.

% ----------------------------------------------------------------
% Paste the main body of the paper here.
% Remove any occurrences of \Parasplit.
% Replace any PNAS-only section wrappers as described below.
% ----------------------------------------------------------------

% Example:
% \section{Introduction}
% ...
% \section{The Closure Problem}
% ...
% \section{Data-Driven Closures}
% ...
% \section{Numerical Results}
% ...
% \section{Conclusions}
% ...

Many complex systems in science and engineering involve more dynamically relevant degrees of freedom than can be feasibly resolved. Practical models therefore evolve only a subset of variables while treating the remainder implicitly. As a result, they are necessarily \textit{coarse-grained}: they retain a reduced set of variables while neglecting unresolved components.

This simplification introduces model error with distinctive structure. Unlike measurement noise, the error induced by coarse graining is dynamical: it arises from deterministic interactions between resolved and unresolved variables. In chaotic systems, where trajectories diverge exponentially and invariant measures govern long-term statistics, even small systematic misrepresentations of these interactions can accumulate, leading to biased forecasts, incorrect variability, and distorted stationary distributions. The resulting discrepancies are not confined to short-term predictive skill; they fundamentally alter the statistical behavior of the model.

The problem of representing unresolved dynamics, commonly referred to as closure or parameterization, has a long history, particularly in fluid dynamics~\cite{pope2000} and climate science~\cite{christensen2022, berner2017}. Classical closures posit deterministic relationships between the resolved state and the effect of neglected components. More recently, data-driven methods have replaced hand-crafted parameterizations with flexible function approximators trained on high-resolution simulations or observations. These approaches have achieved impressive reductions in one-step prediction error and have demonstrated that machine learning models can flexibly represent how the influence of unresolved dynamics depends on the resolved state.

Most existing data-driven closures are trained in an offline manner: parameters are optimized to minimize discrepancies between predicted and observed tendencies at a single time step. This procedure implicitly treats the coarse-grained dynamics as Markovian. In complex chaotic systems, this assumption is generally invalid due to the absence of a clear timescale separation between the resolved and unresolved degrees of freedom. For example, in fluid turbulence
it has long been recognized that coarse graining introduces non-Markovian memory~\cite{kraichnan1959}. While the strength of this non-Markovianity depends on the particular system, it is known to be significant in geophysical turbulence, where the coherence of baroclinic eddies ensures a strong, history-dependent coupling between the resolved and unresolved scales~\cite{dijkstra2022eddy,vanderborght2024physics}.
More broadly, in any partially resolved system, the effective dynamics of the resolved variables inevitably inherit both memory and stochasticity from the unresolved components~\cite{zwanzig2001nonequilibrium, chorin2000optimal}. Offline training systematically fails to capture this memory. The accumulation of small one-step biases can severely degrade trajectory statistics, often leading to unphysical bias and/or numerical instability.

To correct these temporal errors and implicitly account for this missing memory, models are increasingly calibrated over multi-step trajectories (online learning). In this work, we demonstrate that attempting to resolve these issues deterministically introduces a fundamental mathematical degeneracy. We prove that when the resolved dynamics are intrinsically stochastic due to coarse graining, calibration over trajectories using deterministic distance-based metrics forces a collapse of the system's natural variance. Calibration must instead be framed probabilistically; we show that optimizing strictly proper scoring rules avoids this degeneracy, ensuring that the theoretical optimum is consistent with the system's true conditional law. Crucially, this improvement stems from aligning the learning objective with the stochastic nature of coarse-grained dynamics.

A related phenomenon has been observed in recent data-driven weather forecasting. Deterministic models such as GraphCast~\cite{lam2023} and FourCastNet~\cite{pathak2022fourcastnet} tend to lose fine-scale variability at long lead times. This has motivated growing interest in probabilistic and generative forecasting systems, including GenCast~\cite{price2023gencast} and hybrid models such as NeuralGCM~\cite{kochkov2024}. Our results explain why that shift reflects the structure of the problem, rather than merely an ad hoc engineering choice.

We validate these theoretical insights using quasi-geostrophic turbulence as a canonical chaotic system. We show that offline-trained closures become severely miscalibrated over multi-step trajectories, often leading to numerical instability, while online-trained deterministic closures suffer the variance collapse predicted by our analysis. In contrast, online-trained probabilistic closures overcome both limitations, yielding skillful ensemble forecasts and realistic stationary spectra. These findings establish that stochasticity and trajectory-based optimization are not independent optional refinements, but coupled, essential mathematical requirements for representing unresolved dynamics in chaotic systems, with implications extending far beyond turbulence modeling.

\section{The Closure Problem}
Consider a dynamical system
\begin{equation}\label{eq:complete-system}
    \bm{x}_{n+1}=\bm{\Phi}(\bm{x}_n).
\end{equation}
Let $(\overline{\,\cdot\,})$ denote a coarse-graining operator, mapping $\bm{x}$ to a lower-dimensional representation $\overline{\bm{x}}$, and let $\bm{x}'$ denote the degrees of freedom neglected by $\overline{\bm{x}}$. In particular, assume there is a function $\bm{g}$ such that $\bm{x}$ can be reconstructed from $\overline{\bm{x}}$ and $\bm{x}'$ via $\bm{x}=\bm{g}(\overline{\bm{x}},\,\bm{x}')$. Suppose we want to construct a model for $\overline{\bm{x}}$ without solving for $\bm{x}'$. Then we will refer to $\overline{\bm{x}}$ as the resolved component and $\bm{x}'$ as the unresolved component.

The dynamics of $\overline{\bm{x}}$ and $\bm{x}'$ are in general coupled, so that we have, for some $\bm{\Phi}_1$ and $\bm{\Phi}_2$,
\begin{subequations}
    \begin{align}\label{eq:coarse-grained}
        \overline{\bm{x}}_{n+1} &= \bm{\Phi}_1(\overline{\bm{x}}_n,\,\bm{x}'_n),\\
        \bm{x}'_{n+1} &= \bm{\Phi}_2(\overline{\bm{x}}_n,\,\bm{x}'_n).
    \end{align}
\end{subequations}
Notice $\bm{\Phi}_1(\overline{\bm{x}}_n,\, \bm{x}'_n) = \overline{\bm{\Phi}(\bm{x}_n)}$.

Now suppose we have a simplified model for the evolution of $\overline{\bm{x}}$ in the form of a map $\overline{\bm{x}}\mapsto\overline{\bm{\Phi}}(\overline{\bm{x}})$. Often a choice of $\overline{\bm{\Phi}}$ arises from simply neglecting any terms in $\bm{\Phi}$ which depend on $\bm{x}'$. However, we make no assumptions here about $\overline{\bm{\Phi}}$ or its accuracy. For any $\overline{\bm{\Phi}}$ we can define the model error
\begin{equation}
    \bm{m}(\overline{\bm{x}}_n,\,\bm{x}'_n) = \bm{\Phi}_1(\overline{\bm{x}}_n,\, \bm{x}'_n) - \overline{\bm{\Phi}}(\overline{\bm{x}}_n).
\end{equation}
Correspondingly, we have
\begin{equation}
    \overline{\bm{x}}_{n+1} = \overline{\bm{\Phi}}(\overline{\bm{x}}_n) + \bm{m}(\overline{\bm{x}}_n,\, \bm{x}'_n).
\end{equation}
The model error accounts for the influence of $\bm{x}'_n$ on the evolution of $\overline{\bm{x}}_n$, which is not accounted for by $\overline{\bm{\Phi}}$, as well as any other error introduced by $\overline{\bm{\Phi}}$.

In order to account for the effect of the unresolved dynamics (i.e. the dynamics of $\bm{x}'$), we must somehow approximate or estimate the model error and form a parameterized model
\begin{equation}\label{eq:param_dyn}
    \widetilde{\bm{x}}_{n+1} = \overline{\bm{\Phi}}(\widetilde{\bm{x}}_n) + \widetilde{\bm{m}}_n.
\end{equation}
This very general problem is ubiquitous across science and engineering, where employing a truly complete model is impractical. It is known variously as either parameterization or closure. 

In classical approaches deterministic approximations to the model error are hypothesized in the form
\begin{equation}
    \bm{m}(\overline{\bm{x}}_n,\, \bm{x}'_n) \approx \widetilde{\bm{m}}_n = \bm{f}(\overline{\bm{x}}_n),
\end{equation}
for some function $\bm{f}$. In this way deterministic closures (or parameterizations) assume that the model error can be diagnosed with certainty from the instantaneous value of the resolved state alone. Since $\bm{\Phi}_1$ depends also on $\bm{x}'_n$, this is clearly not true exactly. Rather $\overline{\bm{x}}_n$ provides some limited information about $\bm{m}_n$, but uncertainty remains so long as $\bm{x}'_n$ is unknown. Intuitively, given a value of $\overline{\bm{x}}_n$, there are generally many compatible (physically realizable) values of $\bm{x}'_n$, and each of these values causes $\overline{\bm{x}}$ to evolve along a different trajectory.

Stochastic closures represent this uncertainty explicitly by modeling the model error (and hence the future of the resolved state, via~\eqref{eq:param_dyn}) probabilistically. Given a stochastic closure, the uncertainty induced in finite-time predictions by uncertain model error can be quantified by ensemble forecasting. However, the benefit of stochastic closures is not limited to allowing uncertainty quantification as an optional extra beyond deterministic predictions. Rather, as we argue in this work, the probabilistic setting of stochastic closures allows for calibration procedures that are aligned with the inherently uncertain dynamics of the coarse-grained state.
Put differently, there is an inherent conflict in trying to fit deterministic models to match the stochastic evolution of coarse-grained chaotic systems.

\begin{figure*}[t!]
\centering
\includegraphics[width=\textwidth]{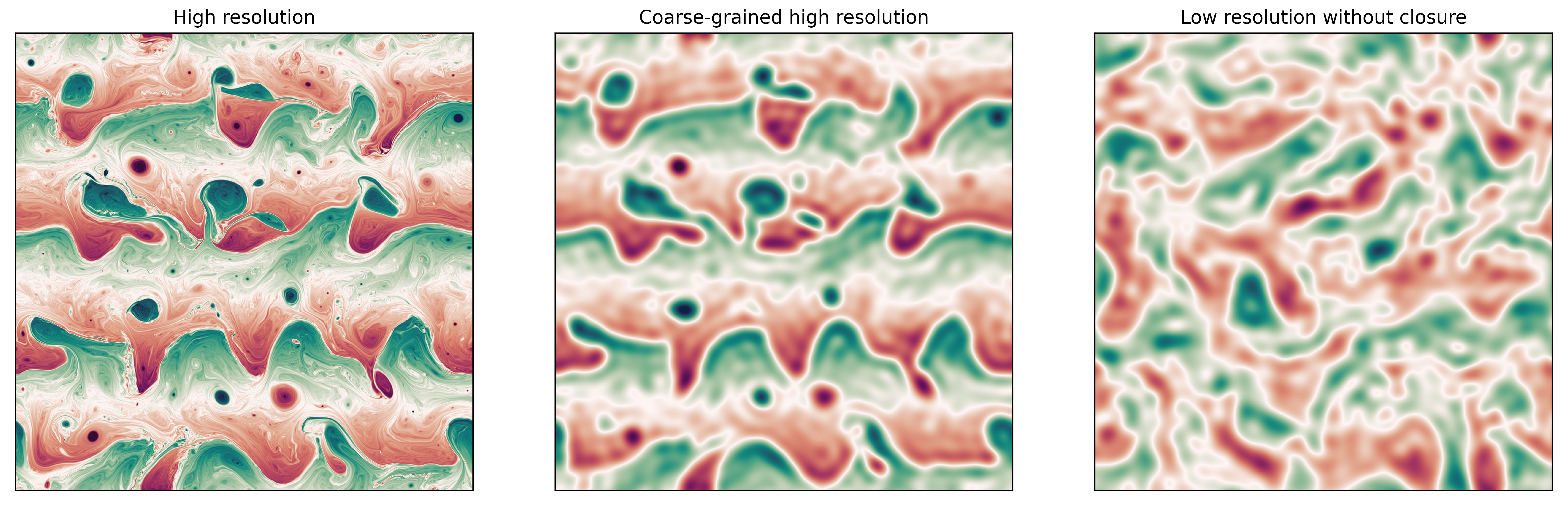}
\caption{Snapshots of upper-layer potential vorticity anomaly in the quasi-geostrophic model at high resolution (left), after coarse graining (center), and from a free-running low-resolution simulation without parameterization (right).}
\label{fig:coarse-graining}
\end{figure*}

\section{Data-Driven Closures}
Constructing closures that accurately reflect the uncertain future of the coarse-grained state is challenging. The goal of probabilistic forecasting is to design a stochastic closure such that realizations of the parameterized model~\eqref{eq:param_dyn} are consistent with the conditional law of future coarse-grained trajectories
\begin{align}\label{eq:conditional_law}
    \mathbb{P}(\overline{\bm{x}}_{n+1},\,\dots,\overline{\bm{x}}_{n+m}\mid\overline{\bm{x}}_n=\bm{x}^*),
\end{align}
which reflect how likely future trajectories are given a coarse-grained initial condition $\bm{x}^*$.
The difficulty is that these conditional probabilities usually do not satisfy assumptions that would simplify their representation in models. In particular, the coarse-grained process $\{\overline{\bm{x}}_n\}$ is generally non-Markovian. This means that a factorization of~\eqref{eq:conditional_law} into a product of one-step transition probabilities $\prod_i\mathbb{P}(\overline{\bm{x}}_{n+i+1}\mid\overline{\bm{x}}_{n+i})$ is not valid. 

\subsection*{Offline vs. Online learning}
In standard approaches, non-Markovianity is ignored; a model of the one-step transition probability is constructed, and this is sampled iteratively at each time step to produce predictions over multiple time steps. We refer to this approach as \textit{offline learning}. Deterministic closures can be considered a special case, where the one-step transition probability has zero variance. With some more recent exceptions, most existing work on data-driven closure modeling has followed the offline-learning framework~\cite{rasp2018deep, brenowitz2019spatially, yuval2020stable, gagne2020machine, zanna2020data, guillaumin2021, brolly2025a}. 

A general Markovian approximation $\{\widetilde{\bm{x}}_n\}$ to the coarse-grained process is determined by its own one-step transition probability $\widetilde{\mathbb{P}}(\widetilde{\bm{x}}_{n+1}\mid\widetilde{\bm{x}}_n=x^*)$, which is typically different from that of the coarse-grained process. Consider a parametric family of Markov kernels $\mathbb{Q}_{\bm{\theta}}(\widetilde{\bm{x}}_{n+1}\mid\widetilde{\bm{x}}_n=x^*)$ with parameters $\bm{\theta}$. When the aim of predictions is to approximate multi-step conditional probabilities as in~\eqref{eq:conditional_law}, this involves optimizing an objective which depends on predictions over multiple steps of the parameterized model. This trajectory-based approach is known as \textit{online} or \textit{a posteriori learning}~\cite{rasp2020coupled, kochkov2021machine, frezat2022, maddison2026}.

\subsection*{Scoring rules}

The choice of an objective (or loss function) for calibrating data-driven closures is critical. For deterministic closures, loss functions based on mean squared error (MSE) are most common. However, given that the coarse-grained process is intrinsically stochastic, it is more appropriate to evaluate predictions using scoring rules for probabilistic forecasts~\cite{gneiting2007}. This has long been recognized in weather forecasting. An extensive literature now exists developing scoring rules and establishing their properties~\cite{good1952rational, gneiting2007, brocker2007scoring, du2021beyond}. A crucial property of certain scoring rules is strict propriety~\cite{brocker2007scoring}: strictly proper scoring rules are uniquely optimized (in expectation) by the correct probabilistic forecast. A canonical example is the negative log-likelihood, whose minimization mirrors standard maximum likelihood estimation. However, the likelihood is usually intractable in closure modeling due to high dimensionality and nonlinearity, and it cannot be easily estimated from model samples. Fortunately, strictly proper scoring rules exist that can be estimated directly from finite samples~\cite{ferro2014}. The \textit{energy score} has emerged as a highly effective choice, since it (i) is strictly proper, (ii) is applicable to multivariate data, and (iii) admits a simple unbiased estimator.

Given a forecast measure $\nu$ and an observation $\bm{y}$, the energy score is defined
\begin{align}
    \mathcal{S}(\nu,\,\bm{y}) = \mathbb{E}_\nu\|\bm{Y}-\bm{y}\|-\frac{1}{2}\mathbb{E}_\nu\|\bm{Y}-\bm{Y}'\|,
\end{align}
where both $\bm{Y}\sim\nu$ and $\bm{Y}'\sim\nu$ independently. Given a finite ensemble of $S$ samples from $\nu$, $\{\bm{y}^{(1)},\,\dots,\,\bm{y}^{(S)}\}$, the estimator
\begin{align}\label{eq:S_hat}
\widehat{\mathcal{S}}\big(\{\bm{y}^{(s)}\},\,\bm{y}\big)
&= \frac{1}{S}\sum_{s=1}^S \big\|\bm{y}^{(s)} - \bm{y}\big\| \notag \\
&\quad - \frac{1}{2S(S-1)}\sum_{s=1}^S\sum_{s'=1}^S
\big\|\bm{y}^{(s)}-\bm{y}^{(s')}\big\| ,
\end{align}
satisfies $\mathbb{E}_{\nu}[\widehat{\mathcal{S}}]=\mathcal{S}$ for any $S\geq2$.

Using $\widehat{\mathcal{S}}$, we can evaluate the predictions of a stochastic model at a given lead time using $S\geq2$ independent samples. Moreover, the energy score naturally accommodates deterministic models: for a point prediction $\widehat{\bm{y}}$, it reduces to the Euclidean distance $\mathcal{S}(\widehat{\bm{y}},\,\bm{y}) = \|\widehat{\bm{y}}-\bm{y}\|$.

For the purposes of online learning as described above, and given coarse-grained data in the form $\{\overline{\bm{x}}_n\}_{n=0}^N$, we define a loss function
\begin{align}\label{eq:loss}
    L\big(\bm{\theta}\,;\,\{\overline{\bm{x}}_n\}_{n=0}^N\big) = \frac{1}{N} \sum_{j=1}^{N/w} \sum_{n=(j-1)w+1}^{jw} \widehat{\mathcal{S}}\big(\{\widetilde{\bm{x}}_n^{(s)}\}_{s=1}^S,\,\overline{\bm{x}}_n\big),
\end{align}
where $\widetilde{\bm{x}}_{(j-1)w}^{(s)}=\overline{\bm{x}}_{(j-1)w}$ for $j=1,\,\dots,\,N/w$ and for each of the $S$ realizations. Here $w$ denotes the window length measured in coarse-model time steps; when reporting results in hours or days we refer to the corresponding physical duration. For stochastic closures ($S \ge 2$), $\widehat{\mathcal{S}}$ is the unbiased estimator defined in~\eqref{eq:S_hat}. For deterministic closures ($S=1$), we extend this notation by evaluating the true energy score functional $\mathcal{S}$ analytically on the exact Dirac measure, yielding the Euclidean distance: $\widehat{\mathcal{S}}\big(\{\widetilde{\bm{x}}_n^{(1)}\},\,\overline{\bm{x}}_n\big) = \|\widetilde{\bm{x}}_n^{(1)}-\overline{\bm{x}}_n\|$. 

The online loss~\eqref{eq:loss} can be applied to any model with parameters $\bm{\theta}$ to assess predictions at all lead times up to a horizon $w$, which we term the \textit{window length}. In practice, the data time series is partitioned into sequential windows of length $w$; for each window, the model is initialized with the true observed state and integrated forward. The energy score is calculated at each lead time, and the average across the window yields the loss. Note that~\eqref{eq:loss} reduces to an offline loss in the case $w=1$.

\subsection*{Generative closures}
A stochastic closure must specify a conditional law for the model error given the resolved state. We adopt a generative formulation in which the model error is produced by a deterministic function of the resolved state and an auxiliary random input. Let $\bm{\xi}_n$ denote an independent noise variable with known distribution (e.g., standard Gaussian), and define
\begin{align}\label{eq:m_tilde}
    \widetilde{\bm{m}}_n = \bm{G}_{\bm{\theta}}(\widetilde{\bm{x}}_n,\, \bm{\xi}_n).
\end{align}
This closure is a conditional generative model: randomness enters only through $\bm{\xi}_n$, while $\bm{G}_{\bm{\theta}}$ is a deterministic function governed by parameters $\bm{\theta}$. Deterministic closures arise as the special case in which $\bm{G}_{\bm{\theta}}$ is independent of $\bm{\xi}_n$.

\begin{figure}
    \centering
    \includegraphics[width=\linewidth]{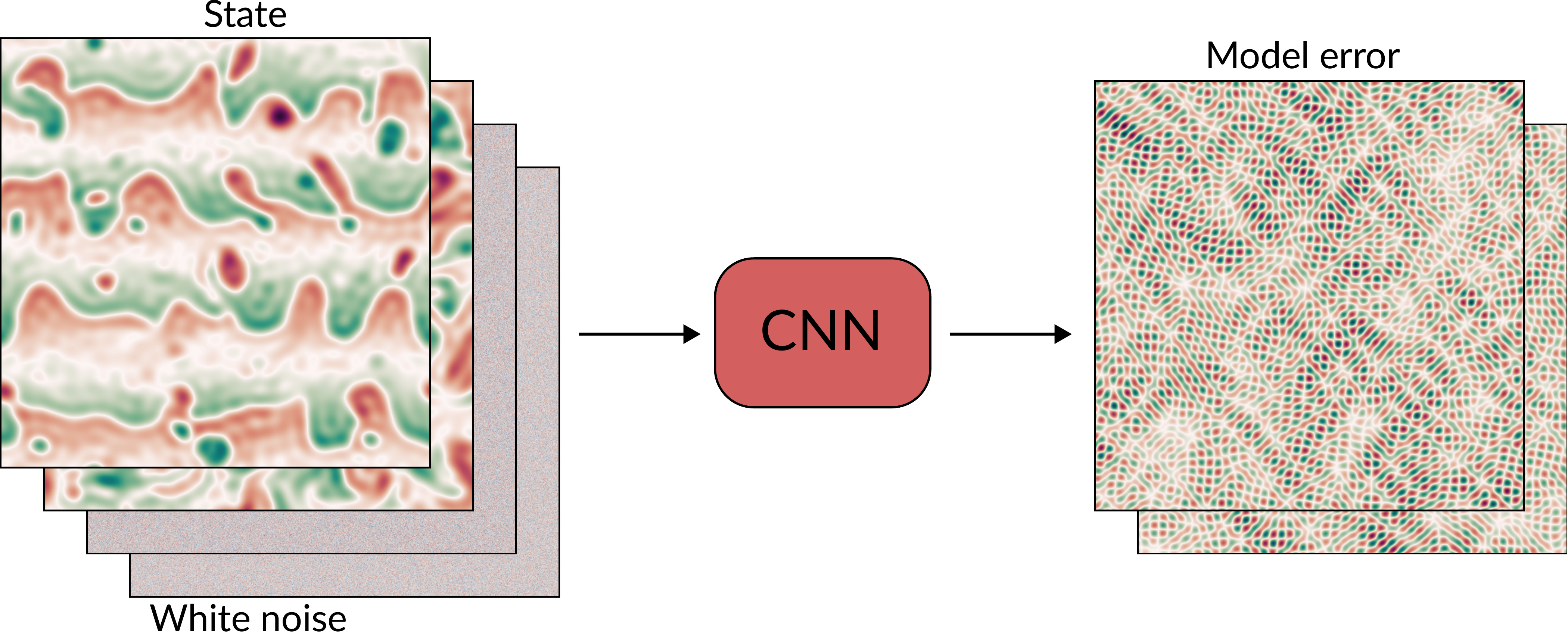}
    \caption{Schematic of the generative closures used in the quasi-geostrophic turbulence experiments. This is an example of~\eqref{eq:m_tilde}, where $\bm{G}_{\bm{\theta}}$ is a convolutional neural network, $\overline{\bm{x}}_n$ is the instantaneous potential vorticity anomaly, and $\bm{\xi}$ is a random field with $\operatorname{dim}\bm{\xi}_n=\operatorname{dim}\overline{\bm{x}}_n$ with each component an independent standard Gaussian random variable.}
    \label{fig:network}
\end{figure}

The primary advantage of coupling generative architectures with sample-based scoring rules is that it bypasses the need for a tractable likelihood. Because the energy score can be evaluated purely from model realizations, we are free to employ highly flexible neural networks to represent the complex conditional distribution of $\widetilde{\bm{x}}_{n+1}$. Furthermore, this formulation places deterministic and stochastic closures within a unified mathematical framework, differing only in whether the learned mapping actively depends on the auxiliary noise.

Importantly, this generative approach does not impose Gaussianity, linearity, or additive structure on the model error. The only structural assumption enforced here is Markovianity. Non-Markovianity of the true coarse-grained dynamics is not represented explicitly in this formulation; rather, trajectory-based calibration encourages the Markovian closure to compensate for missing memory over finite horizons.

\subsection*{Asymptotic properties of trajectory-based objectives}

The core mechanism of online learning is the optimization of an objective function that accumulates loss over a finite time window. Because the parameters $\bm{\theta}$ are shared across steps of the window, a model is optimized when it balances loss across lead times. In chaotic systems with mixing dynamics, however, the true future state $\overline{\bm{x}}_{n+m}$ asymptotically decorrelates from the initial condition $\overline{\bm{x}}_n$. As this decorrelation occurs, precise trajectory tracking becomes impossible. Consequently, loss at large lead times is not governed by short-term forecasting skill, but rather by how the objective function evaluates the model's longer-term predictions against the unpredictable future evolution of the true coarse-grained dynamics. To understand how deterministic learning objectives fail, it is instructive to analyze this asymptotic regime, since even with finite window lengths the imprint of this limit is critical.

We now analyze long-lead training under a common stationary setup. We assume the complete system preserves a probability measure $\mu$ on the full state space. We let $\bar{\mu}$ denote its pushforward onto the coarse-grained space and assume $\bar{\mu}$ has finite second moment. We assume the coarse-grained initial condition is distributed as $\bar{\mu}$ and compare the true coarse-grained dynamics with a Markovian model initialized from the exact resolved state and evolved according to
$$ \widetilde{\bm{x}}_{n+k+1} = \bm{F}_{\bm{\theta}}(\widetilde{\bm{x}}_{n+k},\,\bm{\xi}_{n+k}), \qquad \widetilde{\bm{x}}_n = \overline{\bm{x}}_n, $$ where the noise variables are independent.

\begin{proposition}[Degeneracy of long-lead mean-squared-error training]\label{prop:1}
Assume further that the coarse-grained state decorrelates under the dynamics in the sense that
    \begin{equation*}
        \mathbb{E}\big[\overline{\bm{x}}_{n+m} \mid \overline{\bm{x}}_n\big] \;\xrightarrow{L^2(\bar{\mu})}\; \mathbb{E}_{\bar{\mu}}[\overline{\bm{x}}] \quad \text{as} \quad m \to \infty.
    \end{equation*}
Denote the MSE loss at lead time $m$ as
\begin{equation*}
    \mathcal{J}_m(\bm{\theta}) = \mathbb{E}\big[\| \widetilde{\bm{x}}_{n+m} - \overline{\bm{x}}_{n+m}\|^2\big].
\end{equation*}
This objective admits a decomposition of the form
\begin{equation*}
\begin{split}
    \mathcal{J}_m(\bm{\theta}) &= \mathbb{E}\Big[\big\| \mathbb{E}\big[\widetilde{\bm{x}}_{n+m}\mid \overline{\bm{x}}_n\big] - \mathbb{E}\big[\overline{\bm{x}}_{n+m}\mid \overline{\bm{x}}_n\big]\big\|^2\Big] \\
    &\quad + \mathbb{E}\Big[\operatorname{Tr}\big(\operatorname{Cov}(\widetilde{\bm{x}}_{n+m}\mid \overline{\bm{x}}_n)\big)\Big] + C_m,
\end{split}
\end{equation*}
where $ C_m $ is independent of $ \bm{\theta} $.

In particular, predictive variance enters the MSE loss as an additive nonnegative term. Under the stated decorrelation assumption, the target conditional mean in the first term approaches the climatological mean at long lead times. Thus, the long-lead contribution to the MSE increasingly favors forecasts with vanishing spread and predictive means close to climatology.
\end{proposition}

A detailed proof of this proposition is provided in the SI Appendix. 

The decomposition in Proposition~\ref{prop:1} highlights the mechanism by which deterministic training over extended windows becomes systematically biased. At each lead time, the MSE contains two parameter-dependent contributions: one penalizing the discrepancy between the model's conditional mean and the target's conditional mean, and another directly penalizing predictive variance. As lead time increases, the target conditional mean becomes less dependent on the initial state and converges to the constant climatological mean. The long-lead terms in a time-averaged MSE objective therefore increasingly favor forecasts with suppressed spread and means close to climatology. Because the same parameters govern every step of the model integration, this long-lead pressure feeds back onto the learned dynamics, biasing them toward excess dissipation and helping explain the over-smoothing widely observed in deterministic weather and climate emulators.

\textbf{Remark on Generality:}
Although Proposition 1 is stated for the mean squared error, the SI Appendix shows that the same long-lead collapse holds for any pointwise discrepancy loss, including the Euclidean distance used to train our deterministic baselines. In particular, we show that for any pointwise loss induced by a bivariate discrepancy function $ \ell(\,\cdot\,,\,\cdot\,) $, the asymptotic objective reduces to the expectation of the climatological risk field
$$
R(\widetilde{\bm{y}})
:=
\mathbb{E}_{\overline{\bm{y}} \sim \bar{\mu}}
\big[
\ell(\widetilde{\bm{y}}, \overline{\bm{y}})
\big].
$$
If $ R $ has a unique global minimizer $ \bm{c}^* $, then the long-lead optimum is attained only when the forecast distribution collapses to the Dirac measure at $ \bm{c}^* $. Metric losses are an important special case: when $\ell$ is a metric, $ \bm{c}^* $ is the corresponding Fr\'echet median of the invariant measure. Proper scoring rules are therefore required not merely to replace MSE, but to avoid the broader failure of pointwise trajectory matching.

In contrast to the systematic loss of variance induced by pointwise deterministic training, strictly proper scoring rules remove this structural bias. Assume now that the true lead-time conditional distributions converge to the invariant measure, and that the model's predictive distributions admit a well-defined long-lead limit. Proposition~\ref{prop:2} shows that, under the continuity and integrability assumptions stated in the SI Appendix, the asymptotic expected score decomposes into a constant depending only on the invariant measure and the scoring rule, plus the expected divergence induced by the scoring rule between the model's long-lead predictive distributions and $ \bar{\mu} $. Thus probabilistic trajectory-based training does not penalize predictive spread in itself; it penalizes only distributional mismatch with the true invariant measure.

To state the result, let $ \mathcal{P}_2 $ denote the set of Borel probability measures on the coarse-grained space with finite second moment. For each lead time $ m $, let $ Q_m(\cdot \mid \overline{\bm{x}}_n) $ denote the true lead-time conditional law, let $ \widetilde{\mathbb{P}}_{\bm{\theta},m}(\cdot \mid \overline{\bm{x}}_n) $ denote the model's $ m $-step predictive law, and define the corresponding expected score by
\begin{equation*}
    \mathcal{L}_m(\bm{\theta})
    :=
    \mathbb{E}_{\overline{\bm{x}}_n}
    \bigg[
        \mathbb{E}_{\overline{\bm{y}} \sim Q_m(\cdot \mid \overline{\bm{x}}_n)}
        \Big[
            \mathcal{S}\big(
                \widetilde{\mathbb{P}}_{\bm{\theta},m}(\cdot \mid \overline{\bm{x}}_n),
                \overline{\bm{y}}
            \big)
        \Big]
    \bigg].
\end{equation*}

\begin{proposition}[Asymptotic consistency of strictly proper scoring rules]\label{prop:2}
Under the common stationary setup above, let $ \mathcal{S}(P,\bm{y}) $ be a strictly proper scoring rule on $ \mathcal{P}_2 $.
Assume additionally that, for $ \bar{\mu} $-almost every resolved initial condition $ \overline{\bm{x}}_n $,
\begin{equation*}
    Q_m(\cdot \mid \overline{\bm{x}}_n) \xrightarrow{W_2} \bar{\mu}(\cdot)
\end{equation*}
as $ m \to \infty $, and that the model predictive laws admit a $ W_2 $ limit, denoted by $ \widetilde{\mathbb{P}}_{\bm{\theta},\infty}(\cdot \mid \overline{\bm{x}}_n) $, in the sense that
\begin{equation*}
    \widetilde{\mathbb{P}}_{\bm{\theta},m}(\cdot \mid \overline{\bm{x}}_n) \xrightarrow{W_2} \widetilde{\mathbb{P}}_{\bm{\theta},\infty}(\cdot \mid \overline{\bm{x}}_n).
\end{equation*}
Assume also that the continuity and integrability assumptions stated in the SI Appendix hold.

Then
\begin{equation*}
    \lim_{m\to\infty}\mathcal{L}_m(\bm{\theta})
    =
    \mathbb{E}_{\overline{\bm{x}}_n}
    \bigg[
        d_{\mathcal S}\big(
            \widetilde{\mathbb{P}}_{\bm{\theta},\infty}(\cdot \mid \overline{\bm{x}}_n),
            \bar{\mu}
        \big)
    \bigg]
    +
    \mathbb{E}_{\overline{\bm{y}} \sim \bar{\mu}}
    \big[
        \mathcal{S}(\bar{\mu}, \overline{\bm{y}})
    \big],
\end{equation*}
where $ d_{\mathcal S} $ denotes the divergence induced by $ \mathcal{S} $. Consequently, the parameter-dependent part of the asymptotic objective is exactly the expected divergence between the model's long-lead predictive law and the invariant measure.
\end{proposition}

Here
\begin{equation*}
    d_{\mathcal S}(P,Q)
    :=
    \mathbb{E}_{\bm{y} \sim Q}
    \big[
        \mathcal{S}(P,\bm{y})
    \big]
    -
    \mathbb{E}_{\bm{y} \sim Q}
    \big[
        \mathcal{S}(Q,\bm{y})
    \big]
    \ge 0,
\end{equation*}
with equality if and only if $ P = Q $. The second term in Proposition~\ref{prop:2} is therefore the generalized entropy induced by $ \mathcal{S} $ evaluated at the invariant measure and is independent of the model parameters.

Convergence in the 2-Wasserstein metric ($ W_2 $) combines weak convergence with convergence of second moments. From a physical perspective, this means that agreement at the optimum is not limited to weak distributional agreement, but extends to quadratic quantities such as variance and kinetic energy. If the divergence term in Proposition~\ref{prop:2} vanishes, then the model's long-lead predictive distribution coincides with $ \bar{\mu} $, and its covariance matches the climatological covariance $ \operatorname{Cov}_{\bar{\mu}}(\overline{\bm{x}}) $. More generally, for parametric models, the parameter-dependent part of the asymptotic objective is exactly the expected divergence from $ \bar{\mu} $. Minimizing a strictly proper score therefore drives the model's asymptotic distribution as close as the model class allows to the true invariant measure, rather than toward a variance-collapsed deterministic climatology.

A detailed proof is provided in the SI Appendix.

\section{Numerical results}
We now demonstrate the critical role of stochasticity and online learning in data-driven closure of quasi-geostrophic (QG) turbulence. QG turbulence, developed as an idealized model of large-scale atmosphere/ocean dynamics, is a canonical example of a high-dimensional nonlinear physical system, governed by an active scalar transport-type partial differential equation. Figure~\ref{fig:coarse-graining} shows snapshots of potential vorticity anomaly in the QG model (i) at high resolution (which we consider the complete model system, corresponding to~\eqref{eq:complete-system}), (ii) after applying a coarse-graining operator (corresponding to the coarse-grained dynamics~\eqref{eq:coarse-grained}), and (iii) from an equivalent low-resolution simulation without closure (corresponding to~\eqref{eq:param_dyn} with $\widetilde{\bm{m}}\equiv0$). The dynamics and statistics of the low-resolution model are markedly different from those of the coarse-grained system. We are interested in learning closures from coarse-grained simulations to reduce this discrepancy.

In order to avoid conflation with other approximation/estimation issues, we employ highly flexible generative parameterizations of the form~\eqref{eq:m_tilde} with $\bm{G}_{\bm{\theta}}$ given by a deep convolutional neural network, and use a large amount of data ($100$ years of coarse-grained simulation data) to optimize the network parameters $\bm{\theta}$ with respect to the loss~\eqref{eq:loss}.

We train both deterministic and stochastic closures in this way with various window lengths ($w$ in~\eqref{eq:loss}). Deterministic closures are constructed straightforwardly by training in the same way but without the random input $\bm\xi_n$ in~\eqref{eq:m_tilde} and setting $S=1$, thereby explicitly optimizing the deterministic Euclidean distance reduction of the energy score.
Because the stochastic architecture is structurally identical to this deterministic baseline—differing only by the inclusion of the auxiliary noise field—the significant improvements demonstrated below are attributable solely to the probabilistic training objective rather than added architectural complexity.

In the following sections we evaluate the parameterized models using an additional $100$ years of simulation data not used in training. Full details of the QG model, the coarse-graining operator and the training procedure are given in the Supplementary Information.

\subsection*{Finite-time predictive skill}
We evaluate the finite-time predictive skill of the parameterized models by computing the mean energy score on the validation data.
The validation data is partitioned into 30-day windows. For each window, 20 realizations of each parameterized model are initialized from the true state, and the energy score estimator in~\eqref{eq:S_hat} is computed at lead times from $1$ to $30$ days.
The resulting scores are averaged over the collection of windows. Fig.~\ref{fig:energy_score_vs_w} shows these mean energy scores for the stochastic closures for a range of window lengths $w$. In particular, $w=2$ hours corresponds to a single time step, and therefore offline learning. As $w$ is increased mean energy scores are seen to improve (decrease) across a range of lead times, until at window lengths on the order of several weeks (hundreds of model time steps) the performance plateaus.

\begin{figure}
    \centering
    \includegraphics[width=\linewidth]{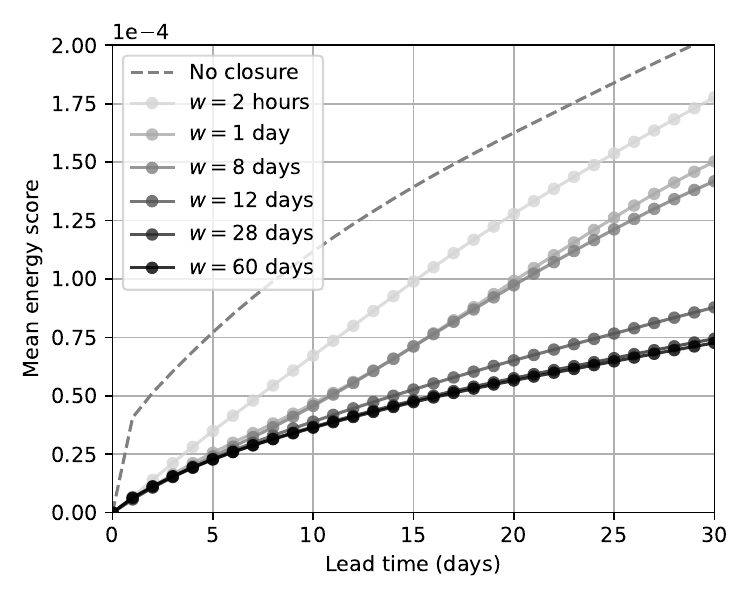}
    \caption{Mean energy score as a function of lead time for stochastic closures trained with various window lengths $w$.}
    \label{fig:energy_score_vs_w}
\end{figure}

In Fig.~\ref{fig:energy_score} mean energy scores are shown for deterministic and stochastic closures, as well as for the coarse model without closure. In particular, the scores shown correspond to closures trained with the value of $w$ which optimizes long-term error in the kinetic energy spectrum, as discussed in the next section. While both deterministic and stochastic closures lead to improvements in forecasting skill, the stochastically parameterized models outperform the deterministic models significantly at all lead times considered.

\begin{figure}
    \centering
    \includegraphics[width=\linewidth]{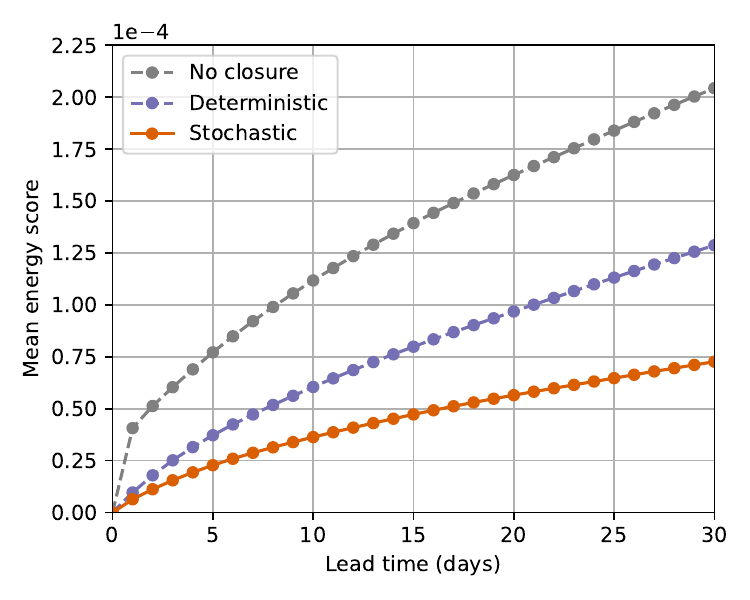}
    \caption{Mean energy score as a function of lead time for online-trained deterministic and stochastic closures, as well as for the coarse model without a closure.}
    \label{fig:energy_score}
\end{figure}

\subsection*{Stationary statistics}
We now consider the long-term consistency of the parameterized models with the coarse-grained dynamics by comparing time-averaged kinetic energy spectra. In particular, we consider the depth-averaged isotropic kinetic energy spectrum, which we denote $\overline{E}(\kappa)$ for the true coarse-grained dynamics and $\widetilde{E}(\kappa)$ in the case of a parameterized model, with $\kappa^2=k_x^2+k_y^2$. We define an error between the two as
\begin{align}
    \Delta E = \frac{1}{\kappa_c}\int_{0}^{\kappa_{\text{c}}}\,\left[\log \frac{\widetilde{E}(\kappa)}{\overline{E}(\kappa)}\right]^2\,\mathrm{d}\kappa,
\end{align}
where $\kappa_{\text{c}}=\frac{2}{3}k_{\text{max}}$ is a cut-off wavenumber. This definition is dimensionless and chosen to be sensitive to errors across a range of wavenumbers.

Fig.~\ref{fig:spec_errors} shows $ \Delta E $ as a function of window length for both deterministic and stochastic closures. The error is also shown for the coarse model without closure for reference. In order to compute these errors, the coarse model was run with each closure for $ 100 $ years. For deterministic closures, window lengths shorter than $ 4 $ days produced models that became unstable before the integrations were completed. For stochastic closures, the same occurred for window lengths shorter than $ 2 $ days, including the offline model at $ w = 2 $ hours. These cases appear as missing points in Fig.~\ref{fig:spec_errors}, since $ \Delta E $ could not be computed once the simulations became unstable. Thus, stochasticity alone is not sufficient to stabilize short-window state-only closures: sufficiently long trajectory-based calibration is still required to accommodate the temporal mismatch between the Markovian closure and the coarse-grained dynamics.

For deterministic closures, an optimal window length appears at an intermediate value of about $ w = 24 $ days, where $ \Delta E $ is noticeably improved relative to the coarse model without closure. This intermediate optimum reflects the compromise imposed by the temporally averaged pointwise loss. At short windows, the model is driven to track realized sample trajectories. As the window is lengthened, the long-lead terms highlighted by the SI generalization of Proposition 1 increasingly penalize predictive spread and deviations from climatology.
In the present setting of QG turbulence, much of the variability being averaged over is concentrated at relatively small scales, where unresolved fluctuations are strongest. The deterministic model is therefore pushed toward a conditional average that preferentially damps those rapidly varying components. Because the same parameters govern every lead time in the window, this long-lead pressure feeds back onto the learned dynamics as a whole, producing excess dissipation and the overly smooth states observed for larger window lengths.

This issue of over-smoothing has been widely observed in the emulation of weather models~\cite{lam2023,bi2023accurate,chantry2025aifs,allen2024assessing} and in parameterization in atmospheric and oceanic modeling~\cite{kochkov2024,maddison2026}. As established in the SI Appendix, this phenomenon is a direct consequence of the degeneracy of any pointwise loss at long lead times, which forces the model to collapse its predictive variance. In contrast, Proposition~\ref{prop:2} shows that strictly proper scoring rules remove this long-lead bias against predictive spread by aligning the asymptotic objective with the invariant measure.

Consistent with this, once the training window is long enough to stabilize the model, the stochastic closures show a continued reduction in $ \Delta E $ as $ w $ is increased, before approaching a broad plateau at the largest window lengths considered. The noise visible at the longest window lengths is practical rather than structural: as $ w $ grows, the fixed $ 100 $-year training record contains proportionally fewer disjoint trajectories of length $ w $, making the online optimization noisier.

\begin{figure}
    \centering
    \includegraphics[width=\linewidth]{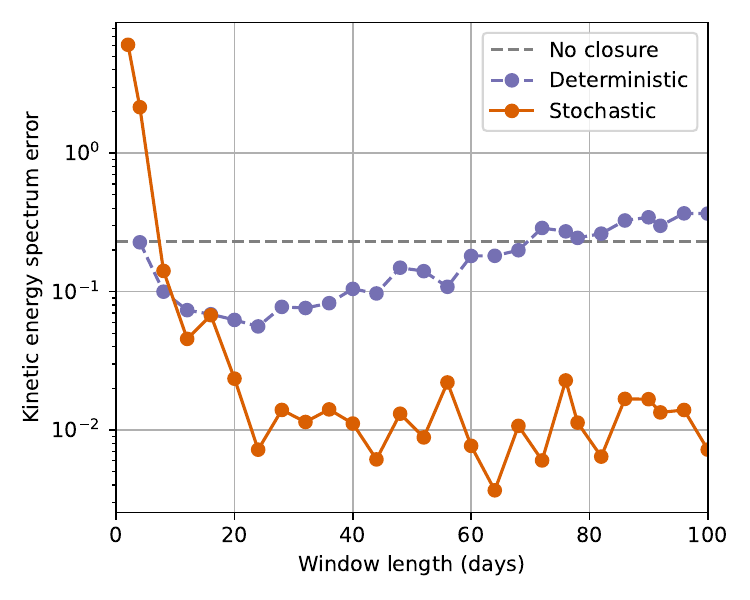}
    \caption{Errors in kinetic energy spectrum as a function of training window length $w$, for deterministic and stochastic closures, as well as without closure.}
    \label{fig:spec_errors}
\end{figure}

In Fig.~\ref{fig:specs} we show the kinetic energy spectrum for the true coarse-grained system, the coarse model without closure, and with deterministic and stochastic closures. For each closure we show the spectrum corresponding to the window length $w$ yielding the lowest spectrum error $\Delta E$. Even with this optimally tuned $w$, the deterministic closure is unable to recover the kinetic energy spectrum accurately, exhibiting a substantial, artificial decrease in energy at scales smaller than approximately $25\,\mathrm{km}$. In contrast, the stochastic closure achieves the correct distribution of energy down to the dissipation scale.

\begin{figure}
    \centering
    \includegraphics[width=\linewidth]{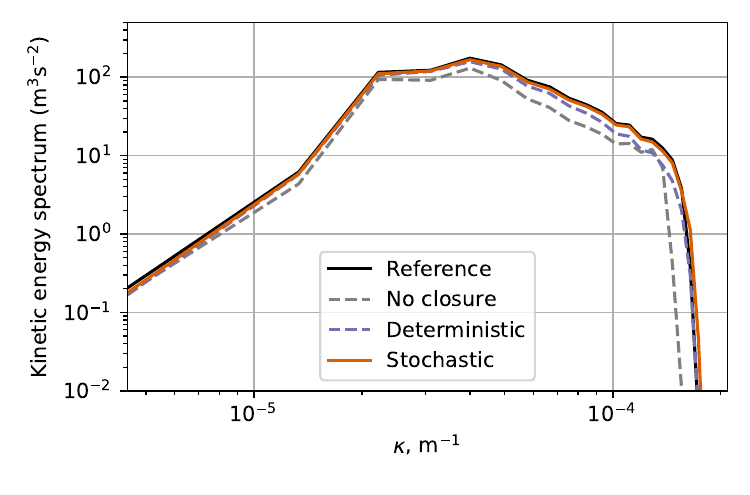}
    \caption{Height-averaged isotropic kinetic energy spectra. Spectra are shown for the true coarse-grained system, the coarse model without closure ($\Delta E = 2.29 \times 10^{-1}$), and the deterministic ($\Delta E = 5.60 \times 10^{-2}$) and stochastic ($\Delta E = 3.67 \times 10^{-3}$) closures. For both closures, results correspond to the window length $w$ yielding the lowest spectrum error.}
    \label{fig:specs}
\end{figure}

To complement the spectral analysis, Fig.~\ref{fig:side} visualizes snapshots of upper-layer PV anomaly fields after $ 100 $ years of simulation with the best online-trained deterministic and stochastic closures, as well as in the true coarse-grained system and at low resolution without any closure. Offline-trained deterministic and stochastic closures are omitted because both became unstable before the end of the $ 100 $-year integrations. Note first the difference between the target process $ (\overline{\bm{x}}_n) $ and the unparameterized coarse simulation (corresponding to $ \widetilde{\bm{x}}_{n+1} = \overline{\bm{\Phi}}(\widetilde{\bm{x}}_n) $): the coherent structure of jets and eddies is lost at low resolution without closure. The online-trained deterministic closure led to overly smooth fields, a direct physical manifestation of the variance-loss tendency described by Proposition~\ref{prop:1}. By contrast, the online-trained stochastic closure sustains a physically realistic flow field that is visually consistent with the true coarse-grained dynamics. This comparison encapsulates the two structural limitations identified above: short-window state-only training fails to accommodate unresolved memory, while deterministic pointwise trajectory training suppresses natural variability. Probabilistic trajectory-based calibration avoids both pathologies and maintains both large-scale coherent structures and small-scale turbulent fluctuations.

\begin{figure}
\centering
\includegraphics[width=\linewidth]{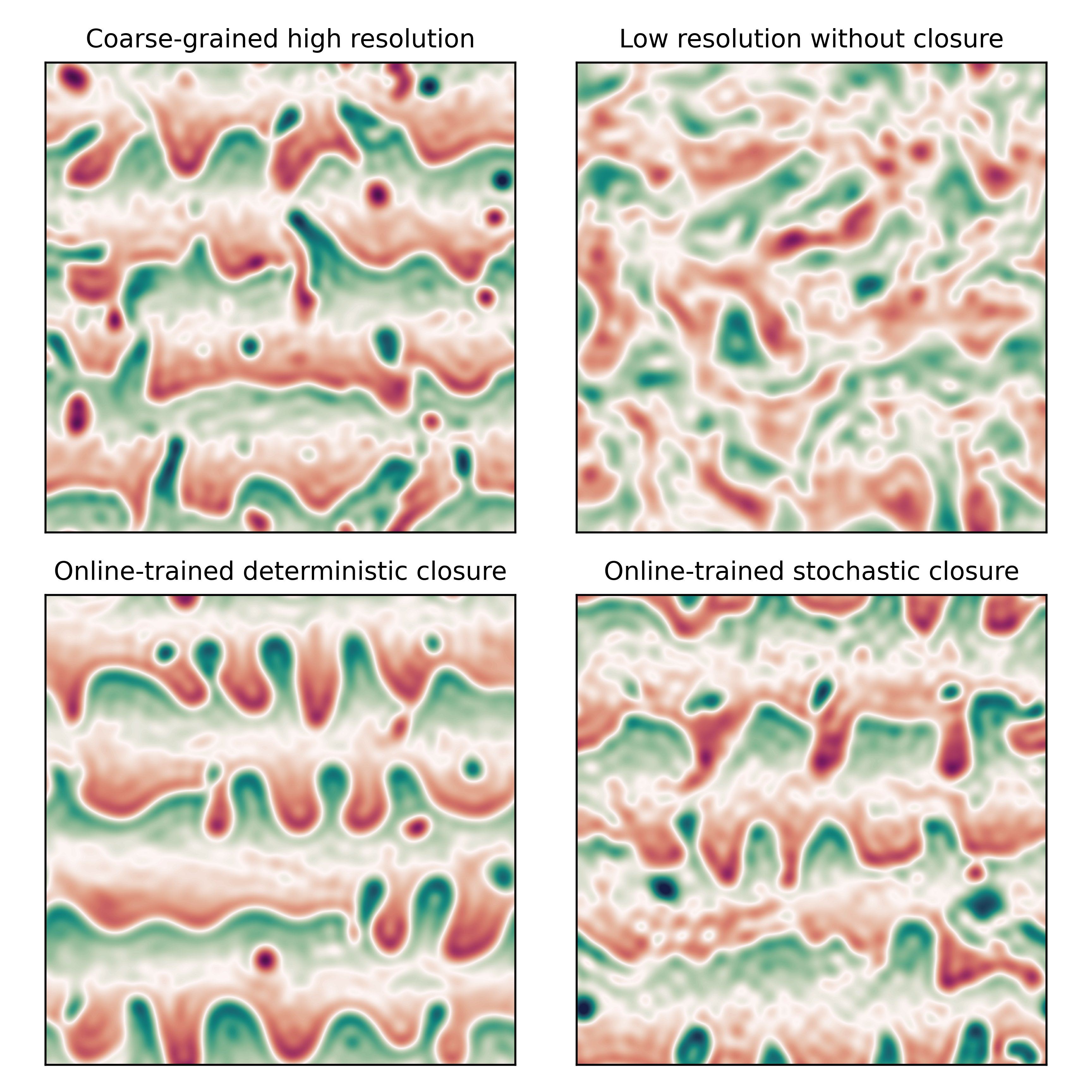}
\caption{Snapshots of upper-layer PV anomaly from free-running simulations after $ 100 $ years with the best online-trained closures, both stochastic and deterministic. Offline-trained closures are omitted because they became unstable before the end of the integrations.}
\label{fig:side}
\end{figure}

\section{Conclusions}

The development of numerical models of complex systems inevitably involves neglecting certain degrees of freedom. A standard approach to calibrating such models is to minimize one-step deterministic prediction errors, such as the mean squared error.
Although evaluating deterministic metrics over extended trajectories is becoming more common, with the aim of improving stability, our analysis demonstrates that doing so introduces a systematic long-lead bias. As established analytically and validated numerically, optimizing any deterministic distance metric over multi-step trajectories inherently forces the model to damp out natural variability, leading to a collapse of predictive variance and distorted long-term statistics.

To represent the dynamics of coarse-grained systems accurately, model calibration must account for the intrinsic uncertainty induced by unresolved components. We have shown that when models are formulated probabilistically and tuned over finite trajectories using strictly proper scoring rules, this structural degeneracy is avoided. Trajectory-based probabilistic scoring aligns the calibration objective with the true dynamics, allowing models to simultaneously capture short-term conditional evolution and the correct long-term invariant measure.

Importantly, these findings extend beyond machine learning parameterizations in fluid dynamics, applying generally to the calibration of any reduced-order model. Whenever a chaotic system is only partially resolved, standard deterministic calibration will predictably degrade long-term fidelity. Trajectory-based probabilistic learning is therefore not merely an optional add-on for uncertainty quantification, but an essential ingredient in building statistically faithful reduced models of coarse-grained systems.

\section*{Materials and Methods}
% Paste the content you currently have inside \matmethods{...} here.

High-resolution data were generated from a two-layer quasi-geostrophic turbulence model on a doubly periodic domain. In the learning problem, the state at time $ t_n $ was the stacked pair of layerwise potential-vorticity-anomaly fields: the full state $ \bm{x}_n $ was defined on the fine grid, and the resolved state $ \overline{\bm{x}}_n $ was obtained by coarse-graining to the low-resolution grid. The coarse-graining operator combined spectral truncation with a low-pass filter. After an initial spin-up to approximate stationarity, 100 years of simulation output were used for training and an independent 100-year dataset was used for validation and evaluation.

Closures were represented as conditional generative models for unresolved model error,
\begin{equation*}
\widetilde{\bm{m}}_n = \bm{G}_{\bm{\theta}}(\widetilde{\bm{x}}_n,\bm{\xi}_n),
\end{equation*}
where $ \bm{G}_{\bm{\theta}} $ was a convolutional neural network and $ \bm{\xi}_n $ was an independent standard Gaussian noise field. Deterministic closures were obtained by removing the noise input. Models were trained either offline with one-step losses or online over finite windows. Stochastic closures were calibrated with the energy score, estimated from model ensembles, and optimization used AdamW with a curriculum in window length. Predictive skill was evaluated using validation energy scores, and long-run fidelity was assessed using time-averaged kinetic-energy spectra from long free-running simulations. Full equations, parameter values, architecture details, and numerical methods are given in the SI Appendix.

\section*{Data availability}
The JAX implementation of the quasi-geostrophic model and all other code required to reproduce the results herein are available on Zenodo (\url{https://doi.org/10.5281/zenodo.19335604}). The trained neural network models are also available on Zenodo (\url{https://doi.org/10.5281/zenodo.19335706}).

\section*{Acknowledgments}
This work was supported by EPSRC grant number EP/X01259X/1. This work was also supported by the Edinburgh International Data Facility (EIDF) and the Data-Driven Innovation Programme at the University of Edinburgh. Access to EIDF was facilitated through the University of Edinburgh’s Generative AI Laboratory GAIL Fellow scheme.

\bibliographystyle{pnas-new}

% \bibliography{main}

{
\footnotesize % Shrinks the text (can also use \small or \scriptsize)
\let\oldbibitem\bibitem
\renewcommand{\bibitem}{%
  \vspace{-0.3em} % Negative vertical space to crush lines together
  \oldbibitem
}

}

%%%%%%%%%%%%%%%%%%%%%%%%%%%%%%%%%%%%%%%%%%%%%%%%%%
% -------------- Appendix ------------------------
%%%%%%%%%%%%%%%%%%%%%%%%%%%%%%%%%%%%%%%%%%%%%%%%%%

\clearpage
\newgeometry{top=1in, bottom=1in, left=1in, right=1in, onecolumn} 
\onecolumn % Call this just to be absolutely safe
\appendix

% Optionally add a big title for the Appendix
\begin{center}
    \LARGE \textbf{Supporting Information}
\end{center}
\vspace{2em}

%%%

\section*{Proofs of main-text results}

% \section*{Proof of Proposition 1: Degeneracy of long-lead mean-squared-error training}

\begin{proof}[Proof of Proposition 1]

Let $\bar{\mu} := \mu \circ \overline{(\cdot)}^{-1}$ be the pushforward invariant measure on the resolved space. By construction, the trajectories of the model are generated by the recurrence $\widetilde{\bm{x}}_{n+k+1} = \bm{F}_{\bm{\theta}}(\widetilde{\bm{x}}_{n+k},\, \bm{\xi}_{n+k})$, starting from the exact initial condition $\widetilde{\bm{x}}_n = \overline{\bm{x}}_n$. Therefore, the prediction at lead time $m$ can be written as a deterministic composition of the initial condition and the accumulated sequence of independent noise variables, which we denote as $\bm{\Xi}_{n,m} = (\bm{\xi}_n,\, \bm{\xi}_{n+1},\, \dots,\, \bm{\xi}_{n+m-1})$, yielding
\begin{equation*}
    \widetilde{\bm{x}}_{n+m} = \bm{F}_{\bm{\theta}}^{(m)}(\overline{\bm{x}}_n,\, \bm{\Xi}_{n,m}).
\end{equation*}

We expand the mean-squared-error objective using the law of total expectation, conditioning on the resolved state $\overline{\bm{x}}_n\sim\bar{\mu}$ to obtain
\begin{equation*}
    \mathcal{J}_m(\bm{\theta}) = \mathbb{E}_{\bar{\mu}} \bigg[ \mathbb{E} \Big[ \big\| \widetilde{\bm{x}}_{n+m} - \overline{\bm{x}}_{n+m} \big\|^2 \;\Big|\; \overline{\bm{x}}_n \Big] \bigg].
\end{equation*}
For any fixed initial condition $\overline{\bm{x}}_n$, we apply the multivariate bias-variance identity to the inner conditional expectation, which gives
\begin{align*}
    \mathbb{E} \Big[ \big\| \widetilde{\bm{x}}_{n+m} - \overline{\bm{x}}_{n+m} \big\|^2 \;\Big|\; \overline{\bm{x}}_n \Big] 
    &= \Big\| \mathbb{E}\big[\widetilde{\bm{x}}_{n+m} \mid \overline{\bm{x}}_n\big] - \mathbb{E}\big[\overline{\bm{x}}_{n+m} \mid \overline{\bm{x}}_n\big] \Big\|^2 \\
    &\quad + \operatorname{Tr}\Big(\operatorname{Cov}\big(\widetilde{\bm{x}}_{n+m} \mid \overline{\bm{x}}_n\big)\Big) 
    + \operatorname{Tr}\Big(\operatorname{Cov}\big(\overline{\bm{x}}_{n+m} \mid \overline{\bm{x}}_n\big)\Big) \\
    &\quad - 2\operatorname{Tr}\Big(\operatorname{Cov}\big(\widetilde{\bm{x}}_{n+m},\, \overline{\bm{x}}_{n+m} \mid \overline{\bm{x}}_n\big)\Big).
\end{align*}
Conditioned on $\overline{\bm{x}}_n$, the surrogate prediction $\widetilde{\bm{x}}_{n+m}$ is a measurable function of $(\overline{\bm{x}}_n,\bm{\Xi}_{n,m})$ and hence its only source of randomness is the injected noise sequence $\bm{\Xi}_{n,m}$. By construction, $\bm{\Xi}_{n,m}$ is generated independently of the true system and, in particular, satisfies the conditional independence $\bm{\Xi}_{n,m}\indep \bm{x}'_n \mid \overline{\bm{x}}_n$. Since the true future $\overline{\bm{x}}_{n+m}$ is a measurable function of $(\overline{\bm{x}}_n,\bm{x}'_n)$, it follows that $\widetilde{\bm{x}}_{n+m}\indep \overline{\bm{x}}_{n+m}\mid \overline{\bm{x}}_n$, and therefore
\begin{equation*}
    \operatorname{Cov}\!\big(\widetilde{\bm{x}}_{n+m},\,\overline{\bm{x}}_{n+m}\mid \overline{\bm{x}}_n\big)=\bm{0}.
\end{equation*}

Taking the outer expectation over the coarse-grained invariant measure $\bar{\mu}$, the total loss separates into three non-negative terms:
\begin{align*}
    \mathcal{J}_m(\bm{\theta}) &= \mathbb{E}_{\bar{\mu}} \bigg[ \Big\| \mathbb{E}\big[\widetilde{\bm{x}}_{n+m} \mid \overline{\bm{x}}_n\big] - \mathbb{E}\big[\overline{\bm{x}}_{n+m} \mid \overline{\bm{x}}_n\big] \Big\|^2 \bigg] \tag{Conditional Bias} \\
    &\quad + \mathbb{E}_{\bar{\mu}} \bigg[ \operatorname{Tr}\Big(\operatorname{Cov}\big(\widetilde{\bm{x}}_{n+m} \mid \overline{\bm{x}}_n\big)\Big) \bigg] \tag{Predictive Variance} \\
    &\quad + \mathbb{E}_{\bar{\mu}} \bigg[ \operatorname{Tr}\Big(\operatorname{Cov}\big(\overline{\bm{x}}_{n+m} \mid \overline{\bm{x}}_n\big)\Big) \bigg]. \tag{Target Uncertainty}
\end{align*}
The third term (Target Uncertainty) depends solely on the true dynamics and represents the irreducible error given the initial condition. Defining $C_m$ as this term, it is independent of the model parameters $\bm{\theta}$, which establishes the exact decomposition stated in the proposition. Because the Conditional Bias and Predictive Variance terms are expectations of non-negative quantities, they act as non-negative additive penalties in the objective.

By assumption, the coarse-grained state decorrelates under the dynamics such that
\begin{equation*}
    \mathbb{E}\big[\overline{\bm{x}}_{n+m} \mid \overline{\bm{x}}_n\big]
    \xrightarrow{L^2(\bar{\mu})}
    \mathbb{E}_{\bar{\mu}}[\overline{\bm{x}}]
    \qquad \text{as } m \to \infty.
\end{equation*}
Therefore, in the exact decomposition above, the target conditional mean appearing in the conditional-bias term approaches the climatological mean at long lead times, while the predictive-variance term remains an additive non-negative penalty. Combined with the exact decomposition, this proves the proposition.
\end{proof}

% \section*{Proof of Proposition 2: Asymptotic consistency of strictly proper scoring rules}

\subsection*{Continuity and integrability assumptions used in Proposition 2}

Throughout this section, let $\mathcal{P}_2$ denote the set of Borel probability measures on the resolved space with finite second moment. For each fixed parameter value $\bm{\theta}$ and lead time $m$, define
\begin{equation*}
    f_{m,\bm{\theta}}(\overline{\bm{x}}_n)
    :=
    \mathbb{E}_{\overline{\bm{y}} \sim Q_m(\cdot \mid \overline{\bm{x}}_n)}
    \Big[
        \mathcal{S}\big(
            \widetilde{\mathbb{P}}_{\bm{\theta},m}(\cdot \mid \overline{\bm{x}}_n),
            \overline{\bm{y}}
        \big)
    \Big].
\end{equation*}

The continuity and integrability assumptions referenced in main-text Proposition~2 are the following:
\begin{enumerate}
    \item The expected score functional
    \begin{equation*}
        (P,Q) \mapsto \mathbb{E}_{\bm{y}\sim Q}\big[\mathcal{S}(P,\bm{y})\big]
    \end{equation*}
    is jointly continuous with respect to $W_2 \times W_2$ on $\mathcal{P}_2 \times \mathcal{P}_2$.

    \item For each fixed $\bm{\theta}$, the sequence
    \begin{equation*}
        \{f_{m,\bm{\theta}}(\overline{\bm{x}}_n)\}_{m\ge 1}
    \end{equation*}
    is uniformly integrable under $\bar{\mu}$.
\end{enumerate}

\begin{proof}[Proof of Proposition 2]

Let $\bar{\mu} := \mu \circ \overline{(\cdot)}^{-1}$ be the pushforward invariant measure on the resolved space, and use the notation introduced above. By assumption, the true lead-time conditional distributions satisfy
\begin{equation*}
    Q_m(\cdot \mid \overline{\bm{x}}_n) \xrightarrow{W_2} \bar{\mu}(\cdot)
\end{equation*}
for $\bar{\mu}$-almost every $\overline{\bm{x}}_n$, and the model predictive distributions admit a $W_2$ limit
\begin{equation*}
    \widetilde{\mathbb{P}}_{\bm{\theta},m}(\cdot \mid \overline{\bm{x}}_n)
    \xrightarrow{W_2}
    \widetilde{\mathbb{P}}_{\bm{\theta},\infty}(\cdot \mid \overline{\bm{x}}_n)
\end{equation*}
for $\bar{\mu}$-almost every $\overline{\bm{x}}_n$.

By joint continuity of the expected score functional with respect to $W_2 \times W_2$,
\begin{equation*}
    f_{m,\bm{\theta}}(\overline{\bm{x}}_n)
    \to
    f_{\infty,\bm{\theta}}(\overline{\bm{x}}_n)
    :=
    \mathbb{E}_{\overline{\bm{y}} \sim \bar{\mu}}
    \bigg[
        \mathcal{S}\Big(
            \widetilde{\mathbb{P}}_{\bm{\theta},\infty}(\cdot \mid \overline{\bm{x}}_n),
            \overline{\bm{y}}
        \Big)
    \bigg]
\end{equation*}
for $\bar{\mu}$-almost every $\overline{\bm{x}}_n$.

Because the sequence $\{f_{m,\bm{\theta}}(\overline{\bm{x}}_n)\}_m$ is uniformly integrable under $\bar{\mu}$ for each fixed $\bm{\theta}$, we may pass the limit through the outer expectation, yielding
\begin{equation*}
    \lim_{m\to\infty}\mathcal{L}_m(\bm{\theta})
    =
    \mathbb{E}_{\overline{\bm{x}}_n}\big[f_{\infty,\bm{\theta}}(\overline{\bm{x}}_n)\big].
\end{equation*}

Every strictly proper scoring rule $\mathcal{S}$ induces a non-negative divergence
\begin{equation*}
    d_{\mathcal S}(P,Q)
    :=
    \mathbb{E}_{\overline{\bm{y}} \sim Q}\big[\mathcal{S}(P,\overline{\bm{y}})\big]
    -
    \mathbb{E}_{\overline{\bm{y}} \sim Q}\big[\mathcal{S}(Q,\overline{\bm{y}})\big]
    \ge 0,
\end{equation*}
with equality if and only if $P=Q$. Rearranging gives
\begin{equation*}
    \mathbb{E}_{\overline{\bm{y}} \sim Q}\big[\mathcal{S}(P,\overline{\bm{y}})\big]
    =
    d_{\mathcal S}(P,Q)
    +
    \mathbb{E}_{\overline{\bm{y}} \sim Q}\big[\mathcal{S}(Q,\overline{\bm{y}})\big].
\end{equation*}

Applying this identity with
\begin{equation*}
    P=\widetilde{\mathbb{P}}_{\bm{\theta},\infty}(\cdot \mid \overline{\bm{x}}_n),
    \qquad
    Q=\bar{\mu},
\end{equation*}
we obtain
\begin{equation*}
    \lim_{m\to\infty}\mathcal{L}_m(\bm{\theta})
    =
    \mathbb{E}_{\overline{\bm{x}}_n}
    \bigg[
        d_{\mathcal S}\big(
            \widetilde{\mathbb{P}}_{\bm{\theta},\infty}(\cdot \mid \overline{\bm{x}}_n),
            \bar{\mu}
        \big)
    \bigg]
    +
    \mathbb{E}_{\overline{\bm{y}}\sim\bar{\mu}}
    \big[
        \mathcal{S}(\bar{\mu},\overline{\bm{y}})
    \big].
\end{equation*}

The second term depends only on the physical system and the choice of scoring rule, so it is independent of the model parameters $\bm{\theta}$. Thus the parameter-dependent part of the asymptotic objective is exactly the expected divergence between the model's long-lead predictive distribution and the invariant measure.
\end{proof}

\section*{Supplementary theoretical result}

% \section*{Generalization of Proposition 1} % Polish metric space version

The degeneracy identified in main-text Proposition~1 is not specific to mean-squared error. It is a structural feature of any objective that evaluates forecasts only through pointwise discrepancies against realized targets, without rewarding predictive variability.

\subsection*{Formal setting}
Let the resolved state space $(\overline{\mathcal{X}}, d)$ be a Polish metric space, fix a reference point $x_0 \in \overline{\mathcal{X}}$, and let $\mathcal{P}_p(\overline{\mathcal{X}})$ denote the space of Borel probability measures $\nu$ on $\overline{\mathcal{X}}$ such that
\begin{equation*}
\int_{\overline{\mathcal{X}}} d(x,x_0)^p\,\mathrm{d}\nu(x)<\infty.
\end{equation*}
Let
$\bar{\mu} := \mu \circ \overline{(\cdot)}^{-1}\in\mathcal{P}_p(\overline{\mathcal{X}})$ be the pushforward invariant measure on the resolved space, and assume $\overline{\bm{x}}_n\sim\bar{\mu}$.

% We adopt the Polish metric space setting because the result is not specific to Euclidean space, but applies equally when the resolved space takes another form, such as a discrete finite element space, or a metric graph. By doing so we also avoid the bookkeeping associated with working in a specific coordinate system.

We adopt the Polish metric space setting to make clear that the result is not a peculiarity of Euclidean state vectors or the standard norm. The argument depends only on the metric geometry of the resolved space, and therefore applies more generally than a coordinate-based formulation. It also avoids the bookkeeping that comes with fixing a particular coordinate system.

\begin{definition}[Pointwise losses]\label{def:pointwise-losses}
Let
\begin{equation*}
\ell : \overline{\mathcal{X}} \times \overline{\mathcal{X}} \to \mathbb{R}
\end{equation*}
be a continuous bivariate loss function such that for all $\nu \in \mathcal{P}_p(\overline{\mathcal{X}})$ and all $\overline{\bm{x}}\in\overline{\mathcal{X}}$,
\begin{equation*}
\int_{\overline{\mathcal{X}}} \big|\ell(\widetilde{\bm{x}},\overline{\bm{x}})\big|\,\mathrm{d}\nu(\widetilde{\bm{x}})<\infty.
\end{equation*}
The associated \textit{pointwise loss} is
\begin{equation*}
\Lambda(\nu, \overline{\bm{x}})
:=
\mathbb{E}_{\widetilde{\bm{x}} \sim \nu}
\big[
\ell(\widetilde{\bm{x}}, \overline{\bm{x}})
\big].
\end{equation*}
For a deterministic forecast $\widetilde{\bm{x}}_d$, this reduces to $\Lambda(\delta_{\widetilde{\bm{x}}_d}, \overline{\bm{x}}) = \ell(\widetilde{\bm{x}}_d, \overline{\bm{x}})$, where $\delta_{\widetilde{\bm{x}}_d}$ denotes the Dirac measure at $\widetilde{\bm{x}}_d$.
\end{definition}

\begin{suppproposition}[Degeneracy of pointwise losses]
\label{prop:si_general_collapse}
Assume that for $\bar{\mu}$-almost every $\overline{\bm{x}}_n$, the true conditional measure $Q_m(\cdot \mid \overline{\bm{x}}_n) := \mathbb{P}(\overline{\bm{x}}_{n+m} \in \cdot \mid \overline{\bm{x}}_n)$ converges in $W_p$ to the invariant measure:
\begin{equation*}
Q_m(\cdot \mid \overline{\bm{x}}_n) \xrightarrow{W_p} \bar{\mu}(\cdot) \qquad \text{as } m \to \infty.
\end{equation*}

Consider a Markovian model $\widetilde{\bm{x}}_{n+k+1} = \bm{F}_{\bm{\theta}}(\widetilde{\bm{x}}_{n+k}, \bm{\xi}_{n+k})$, with $\widetilde{\bm{x}}_n = \overline{\bm{x}}_n$, and let $\widetilde{\mathbb{P}}_{\bm{\theta},m}(\cdot \mid \overline{\bm{x}}_n) \in \mathcal{P}_p(\overline{\mathcal{X}})$ denote its $m$-step conditional predictive measure. Assume that for $\bar{\mu}$-almost every $\overline{\bm{x}}_n$,
\begin{equation*}
\widetilde{\mathbb{P}}_{\bm{\theta},m}(\cdot \mid \overline{\bm{x}}_n) \xrightarrow{W_p} \widetilde{\mathbb{P}}_{\bm{\theta},\infty}(\cdot \mid \overline{\bm{x}}_n).
\end{equation*}

Let $\ell$ be a continuous bivariate loss satisfying the growth bound
\begin{equation*}
|\ell(\widetilde{\bm{y}},\overline{\bm{y}})|
\le
C\Big(1+d(\widetilde{\bm{y}},x_0)^p+d(\overline{\bm{y}},x_0)^p\Big)
\end{equation*}
for some constant $C>0$ and all $\widetilde{\bm{y}},\overline{\bm{y}}\in\overline{\mathcal{X}}$, and assume that the induced expected loss functional
\begin{equation*}
(P, Q) \mapsto \int_{\overline{\mathcal{X}}} \int_{\overline{\mathcal{X}}} \ell(\widetilde{\bm{y}}, \overline{\bm{y}}) \,\mathrm{d}P(\widetilde{\bm{y}}) \,\mathrm{d}Q(\overline{\bm{y}})
\end{equation*}
is jointly continuous with respect to $W_p \times W_p$. Define the climatological risk $R(\bm{c}) := \mathbb{E}_{\overline{\bm{y}} \sim \bar{\mu}} [\ell(\bm{c}, \overline{\bm{y}})]$, and assume that $R$ has a unique global minimizer $\bm{c}^* \in \overline{\mathcal{X}}$.

For each lead time $m$, define the conditional expected loss:
\begin{equation*}
g_{m,\bm{\theta}}(\overline{\bm{x}}_n) := \int_{\overline{\mathcal{X}}} \int_{\overline{\mathcal{X}}} \ell(\widetilde{\bm{y}}, \overline{\bm{y}}) \,\mathrm{d}\widetilde{\mathbb{P}}_{\bm{\theta}, m}(\widetilde{\bm{y}} \mid \overline{\bm{x}}_n) \,\mathrm{d}Q_m(\overline{\bm{y}} \mid \overline{\bm{x}}_n).
\end{equation*}
The corresponding expected loss is $\mathcal{J}_m^{\ell}(\bm{\theta}) = \mathbb{E}_{\overline{\bm{x}}_n} [g_{m,\bm{\theta}}(\overline{\bm{x}}_n)]$. Assume additionally that, for each fixed $\bm{\theta}$, the sequence $\{ g_{m,\bm{\theta}}(\overline{\bm{x}}_n) \}_{m \ge 1}$ is uniformly integrable under $\bar{\mu}$. Then
\begin{equation*}
\lim_{m \to \infty}\mathcal{J}_m^{\ell}(\bm{\theta}) = \mathbb{E}_{\overline{\bm{x}}_n} \left[ \int_{\overline{\mathcal{X}}} R(\widetilde{\bm{y}}) \, \mathrm{d}\widetilde{\mathbb{P}}_{\bm{\theta},\infty} (\widetilde{\bm{y}} \mid \overline{\bm{x}}_n) \right].
\end{equation*}
In particular,
\begin{equation*}
\lim_{m \to \infty}\mathcal{J}_m^{\ell}(\bm{\theta}) \ge R(\bm{c}^*),
\end{equation*}
with equality if and only if
\begin{equation*}
\widetilde{\mathbb{P}}_{\bm{\theta},\infty}(\cdot \mid \overline{\bm{x}}_n) = \delta_{\bm{c}^*}(\cdot)
\qquad \text{for } \bar{\mu}\text{-almost every } \overline{\bm{x}}_n.
\end{equation*}
Thus, optimizing any pointwise loss over long lead times forces collapse onto a deterministic climatological risk minimizer.

\end{suppproposition}

% \section*{Proof of the Generalization of Proposition 1}
\begin{proof}[Proof of Supplementary Proposition 1]
Conditioning on the initial state $\overline{\bm{x}}_n$, the conditional expected loss is given by
\begin{equation*}
g_{m,\bm{\theta}}(\overline{\bm{x}}_n) := \int_{\overline{\mathcal{X}}} \int_{\overline{\mathcal{X}}} \ell(\widetilde{\bm{y}}, \overline{\bm{y}}) \,\mathrm{d}\widetilde{\mathbb{P}}_{\bm{\theta},m} (\widetilde{\bm{y}} \mid \overline{\bm{x}}_n) \,\mathrm{d}Q_m(\overline{\bm{y}} \mid \overline{\bm{x}}_n).
\end{equation*}
The overall objective is given by $\mathcal{J}_m^{\ell}(\bm{\theta}) = \mathbb{E}_{\overline{\bm{x}}_n} [g_{m,\bm{\theta}}(\overline{\bm{x}}_n)]$. 

By assumption, for $\bar{\mu}$-almost every $\overline{\bm{x}}_n$, $Q_m \xrightarrow{W_p} \bar{\mu}$ and $\widetilde{\mathbb{P}}_{\bm{\theta},m} \xrightarrow{W_p} \widetilde{\mathbb{P}}_{\bm{\theta},\infty}$. Joint continuity of the expected loss functional yields $g_{m,\bm{\theta}}(\overline{\bm{x}}_n) \to g_{\infty,\bm{\theta}}(\overline{\bm{x}}_n)$, where
\begin{equation*}
g_{\infty,\bm{\theta}}(\overline{\bm{x}}_n) := \int_{\overline{\mathcal{X}}} \int_{\overline{\mathcal{X}}} \ell(\widetilde{\bm{y}}, \overline{\bm{y}}) \,\mathrm{d}\widetilde{\mathbb{P}}_{\bm{\theta},\infty} (\widetilde{\bm{y}} \mid \overline{\bm{x}}_n) \,\mathrm{d}\bar{\mu}(\overline{\bm{y}}).
\end{equation*}

Since $\{ g_{m,\bm{\theta}}(\overline{\bm{x}}_n) \}_{m \ge 1}$ is uniformly integrable under $\bar{\mu}$, we may pass the limit through the outer expectation: $\lim_{m \to \infty}\mathcal{J}_m^{\ell}(\bm{\theta}) = \mathbb{E}_{\overline{\bm{x}}_n} [g_{\infty,\bm{\theta}}(\overline{\bm{x}}_n)]$. Now define the climatological risk $R(\widetilde{\bm{y}}) := \int_{\overline{\mathcal{X}}} \ell(\widetilde{\bm{y}}, \overline{\bm{y}}) \, \mathrm{d}\bar{\mu}(\overline{\bm{y}})$. By Fubini's theorem and the growth condition on $\ell$,
\begin{equation*}
g_{\infty,\bm{\theta}}(\overline{\bm{x}}_n) = \int_{\overline{\mathcal{X}}} R(\widetilde{\bm{y}}) \, \mathrm{d}\widetilde{\mathbb{P}}_{\bm{\theta},\infty} (\widetilde{\bm{y}} \mid \overline{\bm{x}}_n),
\end{equation*}
and hence
\begin{equation*}
\lim_{m \to \infty}\mathcal{J}_m^{\ell}(\bm{\theta}) = \mathbb{E}_{\overline{\bm{x}}_n} \left[ \int_{\overline{\mathcal{X}}} R(\widetilde{\bm{y}}) \, \mathrm{d}\widetilde{\mathbb{P}}_{\bm{\theta},\infty} (\widetilde{\bm{y}} \mid \overline{\bm{x}}_n) \right].
\end{equation*}

Because $R$ has a unique global minimizer $\bm{c}^*$, $R(\widetilde{\bm{y}}) \ge R(\bm{c}^*)$ for all $\widetilde{\bm{y}} \in \overline{\mathcal{X}}$, with equality if and only if $\widetilde{\bm{y}} = \bm{c}^*$. Therefore,
\begin{equation*}
\lim_{m \to \infty}\mathcal{J}_m^{\ell}(\bm{\theta}) \ge \mathbb{E}_{\overline{\bm{x}}_n} \left[ \int_{\overline{\mathcal{X}}} R(\bm{c}^*) \, \mathrm{d}\widetilde{\mathbb{P}}_{\bm{\theta},\infty} (\widetilde{\bm{y}} \mid \overline{\bm{x}}_n) \right] = R(\bm{c}^*).
\end{equation*}
Thus $R(\bm{c}^*)$ is the minimum possible value of the asymptotic objective. Equality holds if and only if
\begin{equation*}
\mathbb{E}_{\overline{\bm{x}}_n} \left[ \int_{\overline{\mathcal{X}}} \big( R(\widetilde{\bm{y}}) - R(\bm{c}^*) \big) \, \mathrm{d}\widetilde{\mathbb{P}}_{\bm{\theta},\infty} (\widetilde{\bm{y}} \mid \overline{\bm{x}}_n) \right] = 0.
\end{equation*}
The integrand is nonnegative, so this occurs if and only if the inner integral is zero for $\bar{\mu}$-almost every $\overline{\bm{x}}_n$. Since $R(\widetilde{\bm{y}}) - R(\bm{c}^*) > 0$ for all $\widetilde{\bm{y}} \neq \bm{c}^*$, this is equivalent to $\widetilde{\mathbb{P}}_{\bm{\theta},\infty} (\cdot \mid \overline{\bm{x}}_n) = \delta_{\bm{c}^*}(\cdot)$ for $\bar{\mu}$-almost every $\overline{\bm{x}}_n$. % \hfill $\blacksquare$
\end{proof}

\noindent \textbf{Remark.} Supplementary Proposition 1 applies to a broad class of losses. In particular, it covers deterministic hybrid objectives that augment physical-space mismatch with penalties on transformed features of the state. For the turbulence setting considered in this paper, a natural example is a loss that combines a physical-space error with a discrepancy between Fourier coefficient amplitudes of the forecast and target fields. Provided the resulting bivariate loss satisfies the continuity and growth assumptions above, it remains a pointwise loss in the sense of Definition~\ref{def:pointwise-losses}: it compares a single forecast state with a single realized target state through a deterministic discrepancy. The proposition therefore still applies. Such modifications may change the induced climatological risk field, and may favor deterministic states with more realistic spectral structure, but they do not remove the underlying long-lead degeneracy: the asymptotic optimum is still attained only through collapse onto a deterministic risk minimizer, rather than through a nondegenerate predictive law.

\section*{Quasi-Geostrophic Turbulence Model Details}

We consider a canonical idealized model of geophysical turbulence known as the Phillips model \cite{phillips1954, vallis2017}.
In this two-layer model the quasi-geostrophic potential vorticity (PV) $Q$ is defined for each layer
\begin{subequations}\label{qgpv}
\begin{align}
    Q_1 &= \bm{\nabla}^2\Psi_1 + \beta y + \frac{k_d^2}{1 + \frac{H_1}{H_2}}(\Psi_2 - \Psi_1),\\
    Q_2 &= \bm{\nabla}^2\Psi_2 + \beta y + \frac{H_1}{H_2}\frac{k_d^2}{1 + \frac{H_1}{H_2}}(\Psi_1 - \Psi_2).
\end{align}
\end{subequations}
The parameter $\beta$ arises from the linear approximation of the Coriolis parameter $f=f_0+\beta y$, $H_1$ and $H_2$ are the heights of the upper and lower layers, respectively, and
$k_d^2 = \frac{f_0^2}{g'}\frac{H_1+H_2}{H_1H_2}$, where $g'=g(\rho_2 - \rho_1)/\rho_1$ is the reduced gravity. Up to the effects of friction, $Q$ is materially conserved, so that in the presence of linear bottom friction and small-scale dissipation (denoted ssd), we have
\begin{subequations}\label{qgpv_eq}
\begin{align}
    \frac{\partial Q_1}{\partial t} + J(\Psi_1,\, Q_1) &= \mathrm{ssd},\\
    \frac{\partial Q_2}{\partial t} + J(\Psi_2,\, Q_2) &= \mathrm{ssd} -\gamma \bm{\nabla}^2\Psi_2.
\end{align}
\end{subequations}

We force the system by imposing a base state with vertical shear that is zonal and horizontally uniform. This choice can be motivated physically by either a meridional temperature gradient or zonal wind forcing. In particular, let the imposed mean velocity be $\overline{\bm{U}}_1=(\overline{U}_1,\,0)$ and $\overline{\bm{U}}_2=(0,\,0)$. Equivalently, let
\begin{subequations}
\label{sf_anomaly}
    \begin{align}
        \Psi_1(x,\, y,\, t) &= -\overline{U}_1 y + \psi_1(x,\,y,\,t),\\
        \Psi_2(x,\, y,\, t) &= \psi_2(x,\,y,\,t),
    \end{align}
\end{subequations}
define the stream function anomaly $\psi$. Upon substituting~\eqref{sf_anomaly} into~\eqref{qgpv_eq} and defining the eddy PV
\begin{subequations}\label{qgpv_anomaly}
    \begin{align}
        q_1 &= \bm{\nabla}^2\psi_1 +  \frac{k_d^2}{1 + \frac{H_1}{H_2}}(\psi_2 - \psi_1),\\
        q_2 &= \bm{\nabla}^2\psi_2 + \frac{H_1}{H_2}\frac{k_d^2}{1 + \frac{H_1}{H_2}}(\psi_1 - \psi_2),
    \end{align}
\end{subequations}
we find
\begin{subequations}\label{eq:phillips}
    \begin{align}
        \frac{\partial q_1}{\partial t} &+ J(\psi_1,\,q_1) + \beta_1\frac{\partial \psi_1}{\partial x} + \overline{U}_1\frac{\partial q_1}{\partial x} = \mathrm{ssd},\\
        \frac{\partial q_2}{\partial t} &+ J(\psi_2,\,q_2) + \beta_2\frac{\partial \psi_2}{\partial x} = \mathrm{ssd} - \gamma\bm{\nabla}^2\psi_2,
    \end{align}
\end{subequations}
with $(\beta_1,\, \beta_2) = \left(\beta+\frac{k_d^2}{1 + \frac{H_1}{H_2}}\overline{U}_1,\, \beta - \frac{H_1}{H_2}\frac{k_d^2}{1 + \frac{H_1}{H_2}}\overline{U}_1\right)$.
Note from~\eqref{qgpv},~\eqref{sf_anomaly} and~\eqref{qgpv_anomaly}, that the (total) PV in each layer is
\begin{align*}
    Q_i = q_i + \beta_i y.
\end{align*}

We solve~\eqref{eq:phillips} on a rectangular domain with doubly periodic PV anomaly. The model is implemented in JAX using a Fourier pseudospectral method and third-order Adams--Bashforth time stepping. The small-scale dissipation operator, ssd, is defined in the Fourier domain as a highly scale-selective low-pass filter, as first implemented by~\cite{lacasce1996}. Dealiasing is performed approximately, using the Fourier smoothing technique of \cite{hou2007computing}, which retains significantly more small-scale energy than approximate dealiasing with the popular $2/3$ Fourier truncation technique, but benefits from the same computational speed-up with respect to exact $3/2$ dealiasing. In practice, the influence of the Fourier smoothing scheme is negligible, since the small-scale dissipation operator is already highly effective in controlling aliasing error.

Depending on the configuration of the model parameters, the resulting dynamics are characterized by the presence of coherent vortices and/or zonal jets~\cite[see e.g.][]{berloff_Kamenkovich_2019}. The model can also be configured to resemble more the dynamics of the ocean or the atmosphere. Here we adopt a configuration resembling midlatitude ocean dynamics, in which zonal jets are prominent. Physical parameter values are shown in Table~\ref{tab:qg} --- these are the same as the ``jet configuration'' used by~\cite{ross2023}.

\begin{table}
    \centering
    \begin{tabular}{c|c|c}
        Parameter & Symbol & Value\\
        \hline
        Horizontal domain size & $L_x$, $L_y$ & $1000\,\mathrm{km}$, $1000\,\mathrm{km}$\\
        Heights of layers $1$ and $2$ & $H_1$, $H_2$ & $500\,\mathrm{m}$, $5000\,\mathrm{m}$\\
        Mean upper layer zonal velocity & $\overline{U}_1$ &  $0.025\,\mathrm{m\,s}^{-1}$\\
        Meridional gradient of the Coriolis parameter & $\beta$ & $10^{-11}\,\mathrm{s}^{-1}\,\mathrm{m}^{-1}$\\
        Rate of bottom layer friction & $\gamma$ & $7\times10^{-8}\,\mathrm{s}^{-1}$\\
        First Rossby radius of deformation & $r_d = k_d^{-1}$ & $15\,\mathrm{km}$\\
        Coarse (fine) horizontal resolution & $n_x=n_y$ & $64$ ($1024$)\\
        Coarse (fine) time step & $\delta t$ & $2$ $\mathrm{hours}$ ($15$ $\mathrm{minutes}$)
    \end{tabular}
    \caption{Physical parameters of the quasi-geostrophic turbulence model.}
    \label{tab:qg}
\end{table}

\section*{Neural Network Architecture and Training Details}

\subsection*{State representation and closure architecture}
In the quasi-geostrophic example, the resolved state at time step $ n $ is the stacked pair of coarse-grid potential-vorticity-anomaly fields in the two layers,
$$
\overline{\bm{x}}_n
=
\big(
\overline{q}_1(\cdot,t_n),\,
\overline{q}_2(\cdot,t_n)
\big),
$$
represented in physical space on the coarse grid. The closure predicts the unresolved model-error field
$$
\widetilde{\bm{m}}_n = \bm{G}_{\bm{\theta}}(\widetilde{\bm{x}}_n,\bm{\xi}_n),
$$
where $ \bm{\xi}_n $ is an independent standard Gaussian random field with the same spatial dimensions as the resolved state. Deterministic closures are obtained by removing the dependence on $ \bm{\xi}_n $.

The map $ \bm{G}_{\bm{\theta}} $ is implemented as a fully convolutional neural network with six convolutional layers, periodic padding, and width-5 filters. Each hidden layer is followed by batch normalization and a ReLU activation. In addition, each hidden layer is mapped to a two-channel output by a width-1 convolution, and these residual outputs are summed to form the final prediction. This architecture allows corrections from multiple depths of the network to contribute directly to the predicted model error. To enforce basic physical structure, the predicted model-error field is constrained to have zero spatial mean by setting the zero Fourier mode to zero. A high-pass spectral filter $ \mathcal{G}(k_x,k_y) \propto k_x^2 + k_y^2 $ is then applied to suppress spurious low-wavenumber energy. The resulting network contains 622,488 trainable parameters.

\subsection*{Data generation and coarse-graining}
Training and validation data were generated from long integrations of the high-resolution reference model. Starting from a random initial condition, the model was first integrated for 100 years to reach approximate statistical stationarity. A further 100 years of simulation were then used to generate training data by saving the two-layer PV-anomaly field every 2 hours, matching the coarse-model time step and yielding 438,000 consecutive snapshots. An independent 100-year integration was used to generate an equally large validation dataset.

At each saved time, the fine-grid state is the stacked pair of high-resolution PV-anomaly fields. The resolved state is obtained by applying a coarse-graining operator that combines spectral truncation with the same low-pass filter used for small-scale dissipation in the coarse model, following \cite{ross2023}. Thus, both the training data and the learned closure are defined on coarse-grid PV-anomaly fields in physical space, even though coarse-graining and numerical time stepping employ spectral operations. Figure~\ref{fig:coarsegraining} illustrates the effect of spectral truncation and subsequent filtering on representative PV-anomaly snapshots. The dominant large-scale flow features, including coherent zonal jets and mesoscale eddies, are retained by this procedure, although they are not fully resolved on the coarse grid.

\begin{figure}
    \centering
    \includegraphics[height=7in]{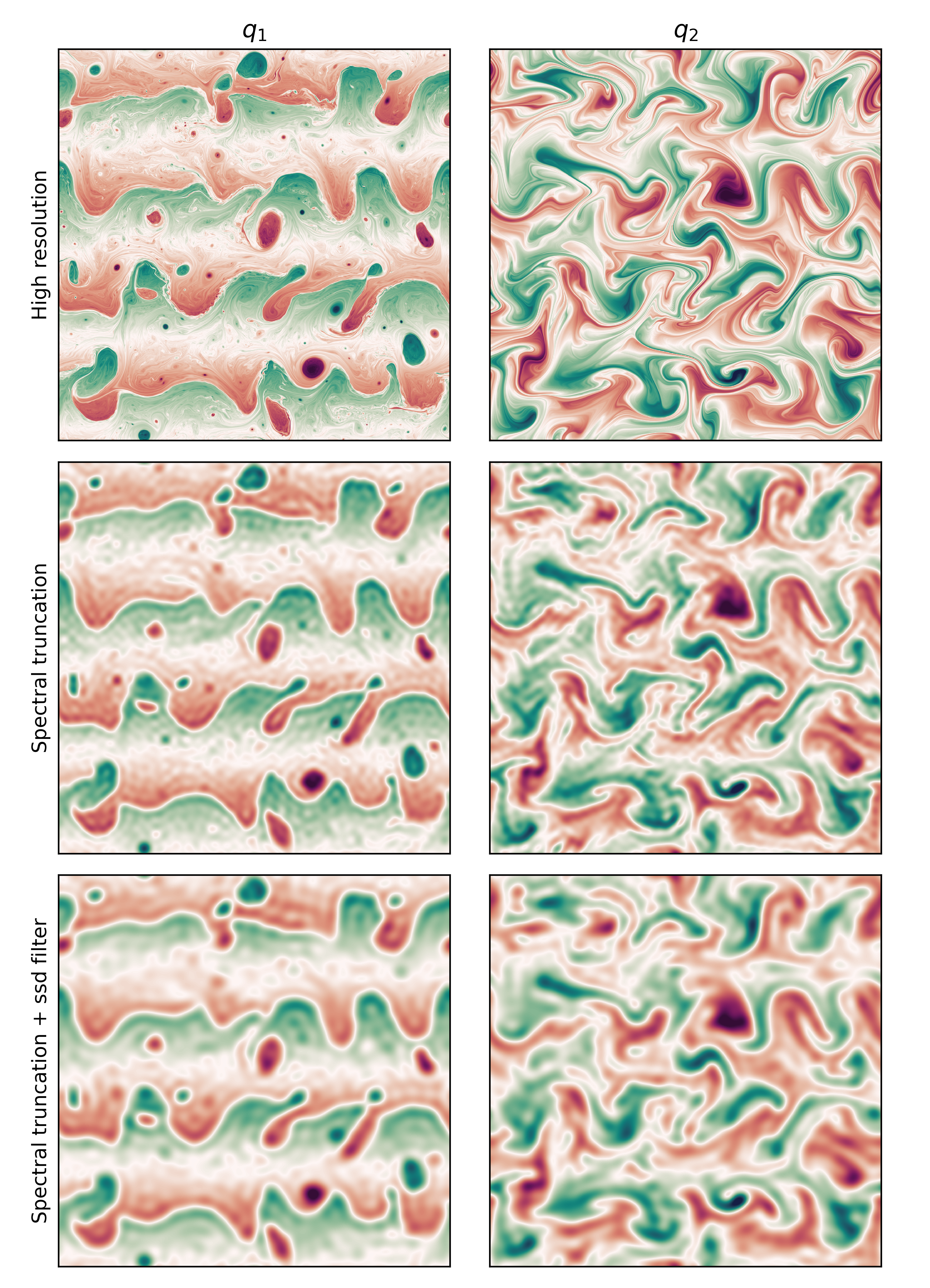}
    \caption{Coarse-graining procedure: snapshots of PV anomaly in the upper layer (left) and lower layer (right) at high resolution (top row), after spectral truncation (middle row), and after spectral truncation followed by application of the low-pass filter used for small-scale dissipation in~\eqref{eq:phillips} (bottom row).}
    \label{fig:coarsegraining}
\end{figure}

\subsection*{Training objective, optimization, and curriculum learning}
Closures were calibrated online using the energy score, which is a strictly proper scoring rule for ensemble forecasts. For a training window of length $ w $, the model is initialized from the true resolved state and integrated forward for $ w $ coarse-model time steps. Parameters $ \bm{\theta} $ are optimized to minimize the empirical mean energy score averaged over lead times within the window. For stochastic closures, this objective is estimated from ensembles of model realizations; for deterministic closures, the corresponding deterministic reduction of the energy score is used.

Optimization employed AdamW together with a cosine-decay learning-rate schedule with warm-up. Training used mini-batches of $ B = 4 $ trajectory windows and an ensemble size of $ S = 4 $ to obtain unbiased gradient estimates for the stochastic objective. To stabilize optimization over long horizons, training followed a 29-phase curriculum in which the window length was progressively increased from one coarse-model time step (2 hours) to 1200 time steps (100 days), with one epoch of training per phase. The independent validation dataset was used only for evaluation and not for parameter updates.

{
\footnotesize 
\let\oldbibitem\bibitem
\renewcommand{\bibitem}{%
  \vspace{-0.3em} 
  \oldbibitem
}
\bibliographystyleS{pnas-sample}

}

\end{document}